\documentclass{article}
\usepackage{amsmath}   
\usepackage{amssymb}   
\usepackage{accents}
\usepackage{graphicx}
\usepackage{color}
\usepackage{tikz}
\usepackage{tikz-cd}
\usetikzlibrary{shapes.geometric,fit,trees}
\usepackage{forest}
\usepackage{paralist}
\usepackage{enumerate}    
\usepackage{amsthm}  
\usepackage[english]{babel}  
\usepackage{textgreek}
\usepackage[normalem]{ulem}
\usepackage{appendix}
\usepackage[shortlabels]{enumitem}

\theoremstyle{theorem}            
\newtheorem{theorem}{Theorem}

\theoremstyle{definition}           

\newtheorem{definition}[theorem]{Definition}

\newtheorem{proposition}[theorem]{Proposition}
\newtheorem{lemma}[theorem]{Lemma}
\newtheorem{corollary}[theorem]{Corollary}

\newtheorem{question}[theorem]{Question}

\newtheorem{remark}[theorem]{Remark}

\newcommand{\ignore}[1]{}
\renewcommand{\phi}{\varphi} 
\renewcommand{\epsilon}{\varepsilon} 

\def\N{\mathbb{N}} 

\def\baire{\mathcal{N}}

\newcommand{\ccX}{\mathcal{X}}

\def\omegack{\omega_1^{\text{ck}}}

\newcommand{\Nats}{\mathbb{N}}

\title{Strongly relativizing reals}
\author{Tyler Arant\footnote{University of California, Los Angeles}}
\date{\today}

\begin{document}

\maketitle

\vspace{-10pt}

\begin{abstract}
For a real $f\in \baire:=\N^\N$, 
it is in general not the case
that every set which is both $\Sigma^1_1(f)$ and $\Pi^1_1(f)$ is the $f$-section
of a $\Delta^1_1$ set.  However, there are reals
$f$ for which this, in fact, does happen; we say that
such a real strongly relativizes $\Delta^1_1$. 
In this paper, we will prove that the reals which
strongly relativize $\Delta^1_1$ are
exactly the hyperlow reals, 
i.e., the $f\in \baire$ with $\omega_1^f=\omegack$.
We will also study reals that strongly relativize
the classes $\Delta^0_\alpha$ for $1\leq \alpha<\omegack$.  We characterize the reals
that strongly relativize $\Delta^0_1$ subsets of $\N$
as the reals which are computably dominated. For $1\leq \alpha<\omegack$, we show that there are continuum many reals which strongly relativize $\Delta^0_\alpha$
subsets of $\N$.  We will also find non-trivial
examples of reals which strongly relativize
$\Delta^0_{\alpha}$ ($1\leq \alpha<\omegack$)
subsets of Baire space.

\end{abstract}

\bigskip

Relativization in descriptive set theory is a way
to translate effective concepts into classical ones.
Working over recursively presented Polish metric spaces,
let $\Gamma$ be one of the (lightface) pointclasses $\Sigma^0_\alpha$, $1\leq \alpha<\omegack$,
or $\Sigma^1_n$, $1\leq n<\omega$.  
If $\mathcal{W}$ is such a space and
$w\in \mathcal{W}$, then $\Gamma$ relativized to $w$, denoted $\Gamma(w)$, is the pointclass consisting of 
all $P\subseteq \mathcal{X}$ such that $P$ is equal to the $w$-section
\[
Q_w:= \{x\in \mathcal{X} : Q(w, x)\}
\]
of some $Q\subseteq \mathcal{W}\times \mathcal{X}$ 
in $\Gamma$.  Such a $P$ has a $\Gamma$ definition which
uses $w$ as a parameter.  
The classical (boldface) pointclass 
corresponding to $\Gamma$,
denoted $\boldsymbol{\Gamma}$, is obtained
by taking the union of all the relativized $\Gamma(w)$ (in fact,
even just by relativizing to all elements of 
the Cantor space $2^\N$).

This type of relativization process works when $\Gamma$ is
one of the
$\Sigma^0_\alpha$, $\Sigma^1_n$ classes (and also works, of course,
for the dual classes $\Pi^0_\alpha$, $\Pi^1_n$); however,
the situation is different for
the self-dual classes $\Delta := \Gamma\cap \neg\Gamma$.
Rather than letting $\Delta(w)$ be the pointclass
of all $w$-sections of $\Delta$ sets,
we define $\Delta$ relativized to $w\in \mathcal{W}$ by
\[
\Delta(w):= \Gamma(w)\cap \neg\Gamma(w).
\]
Then, the corresponding boldface pointclass $\boldsymbol{\Delta}$ is
the union of all the $\Delta(w)$.  Clearly, every 
$w$-section of a $\Delta$ set is in $\Delta(w)$; however,
in general, not every $\Delta(w)$ set is the $w$-section
of a $\Delta$ set.  

Our aim here is to understand for which parameters $w$
it is, in fact, the case that every $\Delta(w)$
set is the $w$-section of a $\Delta$ set.  Throughout,
$\Delta$ will denote any one of $\Delta^0_\alpha$, $1\leq \alpha<\omegack$,
    or $\Delta^1_n$, $1\leq n<\omega$, and 
    we will refer to these classes as the \textbf{self-dual
    Kleene pointclasses}. 
For a self-dual Kleene pointclass $\Delta$, we denote by $\mathcal{S}\Delta(\mathcal{W};w)$ the pointclass
of all $w$-sections of $\Delta$ sets; i.e., $\mathcal{S}\Delta(\mathcal{W};w)$
consists of all pointsets $P\subseteq \mathcal{X}$
such that $P=Q_w$ for some $Q\subseteq \mathcal{W}\times \mathcal{X}$ in $\Delta$.  In this context, we think
of $\mathcal{W}$ as the \textbf{parameter space}.

For a self-dual Kleene pointclass
$\Delta$ and $f\in \baire$, we say that
$f$ \textbf{strongly relativizes $\Delta$
at $\mathcal{X}$} if every $\Delta(f)$ subset of 
$\mathcal{X}$ is in $\mathcal{S}\Delta(\baire; f)$, i.e., if
\[
\mathcal{S}\Delta(\baire; f)\restriction \mathcal{X} = \Delta(f)\restriction \mathcal{X}.
\]
If this holds for all recursively presented Polish metric spaces $\mathcal{X}$, then we simply
say that $f$ \textbf{strongly relativizes $\Delta$}.  

We begin in Section \ref{sec:basic} with some
basic observations about strongly relativizing the
self-dual Kleene pointclasses.  In particular, we will focus
on issues around using $2^\N$ and 
$\baire$ as parameter spaces.

In Section \ref{sec:hyp}, we will completely
characterize the reals which strongly relativize $\Delta^1_1$
(at every space $\mathcal{X}$) as those reals which are
hyperlow.  We will also establish a property of hyperlow reals
that is analogous to a property of computably dominated reals,
which will be further studied in the next section.

Section \ref{sec:recursiveatn} studies strongly relativizing
$\Delta^0_1$ sets.  We will establish
that the elements of $2^\N$ which strongly relativize 
$\Delta^0_1$
are exactly the computably dominated reals in $2^\N$.  
In particular
we will see a close connection between strongly
relativizing $\Delta^0_1$ at $\N$ and truth-table reductions.
We will also show that the $\Pi^0_1$-singletons in $\baire$
strongly relativize $\Delta^0_1$ at $\baire$.  
Putting these results together shows that
the property of strongly relativizing $\Delta^0_1$
subsets of $\baire$ (or even $\N$) is not preserved under
Turing equivalence (on $\baire$).

In Section \ref{sec:arithatn}, we will relativize
some of the results from Section \ref{sec:recursiveatn}
to the classes $\Delta^0_{1+\alpha}$.  Not all
the results relativize, but it is proven that
there are continuum-many reals which strongly
relativize $\Delta^0_{1+\alpha}$ at $\N$.  We
also show that, for any computable ordinal $\alpha>0$, $\Pi^0_{1+\alpha}$-singletons
strongly relativize $\Delta^0_{1+\alpha}$ at every space.

    \bigskip

\noindent \textbf{Notation and terminology.} 
Elements of Baire space $\baire:=\N^\N$ will be denoted by
$f, g, h$, etc.  Elements of Cantor space $2^\N$ will be denoted by $A, B, C,$ etc., thinking of them as subsets of
$\N$, and we will conflate
them with their characteristic functions whenever convenient.
Functions whose domain is Baire space or Cantor space
will be denoted by $F, G, H$, etc. 

We use $\sigma, \tau,$ etc., to denote finite sequences
of integers.  Since elements of $\N^{<\N}$ can be
effectively coded with natural numbers, we consider
$2^{<\N}$ and $\N^{<\N}$ (equipped with the discrete metric) to be recursively
presented Polish metric spaces and freely use them
as arguments in relations.  We write $\sigma\sqsubset \tau$ to mean that $\tau$ strictly extends $\sigma$.
When we write $\sigma\sqsubset A$ for $A\in 2^\N$, we mean
that $\sigma$ is an initial segment of the characteristic function of $A$.

By a \textbf{space}, 
we will always mean a recursively presented Polish metric
space.  See \cite{moschovakis2009} for a development of
the fundamental properties of these spaces.  Spaces will be denoted by $\mathcal{W}, \mathcal{X}, \mathcal{Y}$, etc. 
Every space $\mathcal{X}$ comes with an effective parameterization of basic neighborhoods, which we denote
by $N_s(\mathcal{X})$, $s\in \N$.  Our most important
spaces will be $\baire$, $2^\N$, and $\N$.  For
$\baire$, the basic neighborhoods
can be taken to be $[\sigma] :=\{\alpha\in \baire : \sigma\sqsubset \alpha\}$, and similarly for $2^\N$.

For a pointclass $\Gamma$ and a space $\mathcal{X}$, we denote by $\Gamma\restriction \mathcal{X}$ the collection of
subsets of $\mathcal{X}$ which are in $\Gamma$.  
For a space $\mathcal{X}$, 
$\Sigma^0_1\restriction \mathcal{X}$ is the pointclass
of effectively open subsets of $\mathcal{X}$, i.e., 
those sets $P\subseteq \mathcal{X}$ of the form $\bigcup_{n\in \N}N_{f(n)}(\mathcal{X})$ for a computable $f:\N\rightarrow \N$. This is equivalent to the
the existence of a c.e. relation $R\subseteq \N$ such that
for all $x\in \mathcal{X}$,
\[
P(x) \iff (\exists s)[x\in N_s(\mathcal{X}) \ \& \ R(s)].
\]
If $\mathcal{X}$ is $\baire$ or $2^\N$, we can take
$R$ to be computable.  
The standard normal form for a $\Sigma^0_1$ set $P\subseteq \mathcal{X}\times \mathcal{Y}$ is
\[
P(x, y) \iff (\exists s,t)[x\in N_s(\mathcal{X})\ \& \
y\in N_t(\mathcal{Y})\ \& \ R(s, t)]
\]
for some c.e. relation $R\subseteq \N^2$

An ordinal $\alpha$ is \textbf{computable} if
there is a computable wellordering on $\N$ which
has the same order type as $\alpha$. The 
\textbf{Church-Kleene
ordinal}, denoted $\omegack$, is the least (countable)
ordinal which is not computable.  
Briefly, the pointclass $\Sigma^0_\alpha$ consists of
the Borel sets that have a computable 
(additive) Borel code
of rank $\alpha$.  See \cite[Section 7B]{moschovakis2009} 
for details.
$\Delta^1_1$ is the pointclass of effectively Borel
sets, i.e., sets which
have  computable Borel codes.

For a computable ordinal $\alpha$ and $f\in \baire$,
we denote by $f^{(\alpha)}$ the $\alpha$ Turing jump
of $f$, i.e., we iterate the Turing jump along some
computable presentation of $\alpha$.  By a famous theorem
of Spector (see \cite[Theorem 2.3.1]{chong2015}), up to Turing equivalence 
it does not matter which computable presentation of 
$\alpha$ we use. Note that $f^{(0)}=f$.  

For $f, g\in \baire$, $f\oplus g$ is the join, whose
characteristic function is given by $(f\oplus g)(2n)=f(n)$ and $(f\oplus g)(2n+1) = g(n)$.

\bigskip

\noindent \textbf{Acknowledgments.}  
I would like to thank
Alexander Kechris and Andrew Marks for
their insights and comments on this work.
Patrick Lutz provided very helpful comments and corrections,
and resolved an interesting 
question that the author could not (see Remark
\ref{remark:lutz} and
the appendix).

\section{Basic properties of strongly relativizing}

\label{sec:basic}

Typically, it is most natural to use reals
(i.e., elements of $\baire$) as parameters in definitions.
For many arguments, using elements 
of Cantor space $2^\N$ as parameters
is even more convenient, since $2^\N$ is (effectively) compact. 
Of course, a real $A\in 2^\N$ is also an element
of $\baire$, so for a self-dual Kleene pointclass
we have two pointclasses
to consider: $\mathcal{S}\Delta(2^\N; A)$
and $\mathcal{S}\Delta(\baire; A)$.  
We always have $\mathcal{S}\Delta(\baire; A)
    \subseteq \mathcal{S}\Delta(2^\N; A)$ because the natural inclusion mapping from $2^\N$ to $\baire$
    is computable.  For most of
    our self-dual Kleene classes, the reverse inclusion also holds.

\begin{proposition} 
Let $\Delta$ be a self-dual Kleene pointclass other than
$\Delta^0_1$.  For any $A\in 2^\N$, 
$\mathcal{S}\Delta(2^\N; A) = \mathcal{S}\Delta(\baire; A)$
\end{proposition}

\begin{proof}
    We only have to show $\mathcal{S}\Delta(2^\N; A)\subseteq
    \mathcal{S}\Delta(\baire; A)$. 
    Since $2^\N$ is a $\Pi^0_1$ subset of $\baire$, it is also in $\Delta$.  Thus, if
    $P=Q_A$, where
    $Q\subseteq 2^\N\times \mathcal{X}$ is in $\Delta$, then we can
    define $R\subseteq \baire \times \mathcal{X}$ by  
    \[
    R(f, x) \iff f\in 2^\N \ \& \ Q(f, x). 
    \]
    $R$ is in $\Delta$ and $P=R_A$.  
\end{proof}

For $\Delta^0_1$, the classes $\mathcal{S}\Delta^0_1(\baire; A)$ 
    and $\mathcal{S}\Delta^0_1(2^\N; A)$ are also the same when
    we restrict to effectively zero-dimensional spaces.  A recursively presented Polish metric $\mathcal{X}$ space has an
    effective enumeration of basic neighborhoods, $N_s(\mathcal{X})$ for $s\in \N$.  If this collection
    is $\Delta^0_1$ uniformly in $s$, then we say
    $\mathcal{X}$ is \textbf{effectively zero-dimensional}.
    The spaces $\N, 2^\N, \baire$ and their products are all
    effectively zero-dimensional.  

\begin{proposition}
Let $\mathcal{X}$ be effectively zero-dimensional.
Then, for any $A\in 2^\N$, 
$\mathcal{S}\Delta^0_1(2^\N; A) \restriction \mathcal{X} = 
\mathcal{S}\Delta^0_1(\baire; A)\restriction \mathcal{X}$.
\end{proposition}

\begin{proof}
    Assume $P=Q_A$ for some $Q\subseteq 2^\N\times \mathcal{X}$ in $\Delta^0_1$. Using the standard representation for $\Sigma^0_1$ sets, we can find 
$\Sigma^0_1$ sets $Q_0, Q_1\subseteq 2^{<\N}\times \N$ such that
\[
Q(B, x) \iff (\exists \sigma)(\exists s)[\sigma \sqsubset B
\ \& \ x\in N_s(\mathcal{X}) \ \& \ Q_0(\sigma, s)]
\]
\[
\neg Q(B, x) \iff (\exists \sigma)(\exists s)[\sigma \sqsubset B
\ \& \ x\in N_s(\mathcal{X}) \ \& \ Q_1(\sigma, s)]
\]
We may also assume $Q_0, Q_1$ are closed upwards in their
$2^{<\N}$ arguments, i.e., $Q_i(\sigma, s)$ and $\sigma\sqsubset \tau$ implies $Q_i(\tau, s)$.  

We define $R\subseteq \baire \times \mathcal{X}$ as follows.  Given a pair $(f, x)$, use dovetailing to start enumerating (computably in $x$) all
$\sigma\in 2^{<\N}$ for which you find an $s\in \N$
with $x\in N_s(\mathcal{X})$ and a witness to one
of $Q_i(\sigma, s)$, $i=0, 1$.  Note that checking
$x\in N_s(\mathcal{X})$ is $x$-computable because
$\mathcal{X}$ is effectively zero-dimensional.  
By compactness,
at some point you will have enumerated all
the $\sigma$ of a fixed length $k$.  As soon as this happens, if $f\restriction k\in 2^{<\N}$ and if $Q_0(f\restriction k, s)$ was the relation that was witnessed during the search, then set $R(f, x)$ to be
true.  Otherwise, $R(f, x)$ is false.    Clearly,
$R$ is $\Delta^0_1$ and $P=Q_A$.
\end{proof}

We will only be considering strongly relativizing
$\Delta^0_1$ subsets of effectively zero-dimensional
spaces.  Therefore, we can introduce the following
notational convention without any fear of confusion:
for any self-dual Kleene pointclass $\Delta$ and any 
$f\in \baire$ (even if it is an element
of $2^\N$), we use the notation
\[
\mathcal{S}\Delta(f):= \mathcal{S}\Delta(\baire; f),
\]
where in the case that $\Delta=\Delta^0_1$ it
is understood that we are restricting our attention
to subsets of effectively zero-dimensional spaces.

We record some other relevant observations.

\begin{proposition}\label{prop:recsurj}
    Let $\Delta$ be a self-dual Kleene pointclass, and let $\mathcal{X}$ and $\mathcal{Y}$ be spaces such that there is a computable surjection $F:\mathcal{X}\rightarrow \mathcal{Y}$ which 
    has a computable right inverse $G:\mathcal{Y}\rightarrow \mathcal{X}$
    (so that $F\circ G$ is the identity on $\mathcal{Y}$).  If $f\in\baire$ strongly relativizes $\Delta$ at $\mathcal{X}$, then $f$ also
    strongly relativizes $\Delta$ at $\mathcal{Y}$.
\end{proposition}

\begin{proof}
    Suppose $f\in \baire$ strongly relativizes $\Delta$ at $\mathcal{X}$, and let
    $P\subseteq \mathcal{Y}$ be $\Delta(f)$. Define
    $Q:=\{x\in \mathcal{X} : P(F(x))\}$.  $Q$ is clearly a $\Delta(f)$ subset of $\mathcal{X}$,
    hence there is a $\Delta$ relation
    $Q^*\subseteq \baire\times \mathcal{X}$ such that
    $Q=Q^*_f$.  Now, define $P^*\subseteq \baire \times \mathcal{Y}$ by
    \[
    P^*(h, y) \iff Q^*(h, G(y)),
    \]
    which is in $\Delta$.  But then for any $y\in \mathcal{Y}$,
    \[
    P^*(f, y) \iff Q^*(f, G(y)) \iff Q(G(y))
    \iff P(F(G(y))) \iff P(y).
    \]
    Hence, $P=P^*_f$.  
\end{proof}

\begin{corollary} \label{cor:reciso}
Let $\Delta$ be a self-dual Kleene pointclass and
let $f\in \baire$.
\begin{enumerate}
 \item Let $\mathcal{X}, \mathcal{Y}$ be spaces and let $f\in \baire$.  If $\mathcal{X}$ and $\mathcal{Y}$ are computably
isomorphic, then $f$ strongly relativizes $\Delta$ at 
$\mathcal{X}$ if and only if $f$ strongly relativizes
$\Delta$ at $\mathcal{Y}$.
    \item If $f$ strongly relativizes $\Delta$
    at $\baire$, then $f$ strongly relativizes
    $\Delta$ at $\N$.  
\end{enumerate}
\end{corollary}

\begin{proof} 
  (1) immediately follows from Proposition
\ref{prop:recsurj}.  
To prove (2),  just apply Proposition \ref{prop:recsurj} to  the computable surjection 
$\baire\rightarrow \N$, $g\mapsto g(0)$,
and its computable right inverse $n\mapsto (n, 0, 0, \dots)$.
\end{proof}

\section{Strongly relativizing $\Delta^1_1$}

\label{sec:hyp}

The main result of this section is that
the reals which strongly relativize $\Delta^1_1$ (at every
space $\mathcal{X}$) are exactly the hyperlow reals.
So, we begin by defining the hyperlow reals and
recalling some of their key properties.

For $f\in \baire$, a countable ordinal $\alpha$ is 
\textbf{$f$-computable} if there is a $f$-computable wellordering
on $\N$ which whose order type is $\alpha$.
The \textbf{Church-Kleene ordinal of $f$},
denoted $\omega_1^f$, is the least (countable) ordinal
which is not $f$-computable.  Note that 
$\omegack:=\omega_1^\emptyset$. 
$f\in \baire$ is \textbf{hyperlow} if $\omega_1^f=\omegack$.

Kleene's $\mathcal{O}$ is the set of ordinal notations for
computable ordinals.  For our purposes, the most relevant
fact is that $\mathcal{O}$ is a $\Pi^1_1$-complete set
of integers, in the sense that any $\Pi^1_1$ set
of integers can be computably reduced to $\mathcal{O}$.

\begin{theorem}[See \cite{chong2015}, Section 2.4]
    For $f\in 2^\N$, the following are equivalent.
    \begin{enumerate}[(i)]
    \item $f$ is hyperlow.
    \item $\mathcal{O}\notin \Delta^1_1(f)$
    \item For every $\Pi^1_1$ set $P\subseteq \baire$ with $f\in P$, there
    is a $\Delta^1_1$ set $Q\subseteq \baire$ such that
    $f\in Q\subseteq P$.
    \end{enumerate}
\end{theorem}

In addition to these basic facts about hyperlows,
the proof of the main result of this section will
require two other standard pieces of machinery.
First, we recall that $\Pi^1_1$ has the \textbf{reduction
property} (see \cite[4B.10]{moschovakis2009}).  This means that if $P, Q\subseteq \mathcal{X}$
are $\Pi^1_1$, then there exists $\Pi^1_1$ sets $P^*\subseteq P$
and $Q^*\subseteq Q$ such that $P^*, Q^*$ are disjoint and
their union is $P\cup Q$.  We say that the pair $P^*, Q^*$
\textbf{reduces} the pair $P, Q$.  

We will also require a \textbf{good coding} 
for $\Delta^1_1$ subsets of $\mathcal{X}$.
That is to say, a $\Pi^1_1$ set
$D\subseteq \N$ and
sets $Q^\Pi$, $Q^\Sigma\subseteq \N\times \mathcal{X}$ which are $\Pi^1_1$ and $\Sigma^1_1$, respectively,
and satisfy 
\begin{enumerate}[(1)]
\item $e\in D\implies (\forall x\in \mathcal{X})[Q^\Pi(e, x) \iff Q^\Sigma(e, x)]$.
\item For every $\Delta^1_1$ subset $P\subseteq \mathcal{X}$,
there is $e\in D$ such that $P=Q^\Pi_e=Q^\Sigma_e$.  
\end{enumerate}
The construction of such a coding uses the reduction property and is standard in the literature; for example, see \cite[Section 3.2]{hkl}.

\begin{theorem}
\label{thm:hyp}
For $f\in \baire$, the following are equivalent:
\begin{enumerate}[(i)]
\item $f$ is hyperlow; 
\item $f$ strongly relativizes $\Delta^1_1$ (at every space $\mathcal{X}$);
\item $f$ strongly relativizes $\Delta^1_1$ at $\N$. 
\end{enumerate}
\end{theorem}

\begin{proof}
(i)$\Rightarrow$(ii):
Suppose $f$ is hyperlow and
let $P\subseteq\ccX$ be $\Delta^1_1(f)$.
Pick $\Pi^1_1$ sets $P_0, P_1\subseteq \baire\times \ccX$ such that
\[
P(x) \iff P_0(f, x) \iff \neg P_1(f, x)
\]
for all $x\in \mathcal{X}$.
Let $P_0^*$ and $P_1^*$ be $\Pi^1_1$ subsets of
$\baire\times \ccX$ which reduce $P_0$ and $P_1$.
Note that, since $P_{0f}$ and $P_{1f}$ are complements of each other, we have $P=P_{0f} = P_{0f}^*$ and $P=\neg P_{1f}= \neg P_{1f}^*$.
Now, define the set
\[
R:=\{g\in \baire : (\forall x)[P_0^*(g, x) \vee P_1^*(g, x)]\}.
\]
$R$ is clearly $\Pi^1_1$ and $f\in R$.  Note
that for $g\in R$, $P^*_{0g}$ and $P^*_{1g}$ are complements.
Since $f$ is hyperlow,
there is a $\Delta^1_1$ set $S\subseteq \baire$ such that 
$f\in S\subseteq R$.  Now define $Q\subseteq \baire\times \ccX$ by
\[
Q(g, x) \iff g\in S \ \& \  P_0^*(g, x).
\]
$Q$ is clearly $\Pi^1_1$ but it is also $\Sigma^1_1$ since it satisfies
the equivalence
\[
Q(g, x) \iff g\in S \ \& \  \neg P_1^*(g, x).
\]
Finally, it is easy to see that $P=Q_f$, hence $P$ is in $\mathcal{S}\Delta^1_1(f)$.

(ii)$\Rightarrow$(iii) is trivial.

(iii)$\Rightarrow$(i): Suppose $f$ is not hyperlow.  
We show that there exists $B\subseteq \N$
such that $B\in \Delta^1_1(f)$ but $B\notin \mathcal{S}\Delta^1_1(f)$.

Let $D\subseteq\N$ and $Q^\Pi, Q^\Sigma\subseteq \N\times\baire\times\N$ be a good coding
of the $\Delta^1_1$ subsets of $\baire\times \Nats$.  
Since $D\subseteq \Nats$ is $\Pi^1_1$, 
there is a $h\leq_T\mathcal{O}$
which enumerates all the elements of $D$.  
Since $f$ is not hyperlow, $\mathcal{O}\in \Delta^1_1(f)$ and hence $h\in \Delta^1_1(f)$.
Define $B\subseteq \Nats$ by
\[
n\in B\iff \neg Q^\Pi(h(n), f, n).
\]
This definition is clearly $\Sigma^1_1(f)$, but also $B$ is $\Pi^1_1(f)$ 
since $h(n)\in D$ implies that we also have the equivalence
\[
n\in B \iff \neg Q^\Sigma(h(n), f, n).
\]
Suppose by contradiction that $B$ is $\mathcal{S}\Delta^1_1(f)$.  Then, 
there is $e\in D$ such that 
\[
n\in B \iff Q^\Pi(e, f, n).
\]
But for $n$ with $h(n)=e$, we obtain 
\[
Q^\Pi(e, f, n) \iff n\in B \iff \neg Q^\Pi(h(n), f, n) \iff \neg Q^\Pi(e, f, n),
\]
a contradiction.
\end{proof}

Note that there
are continuum-many hyperlow reals, so we have found continuum-many reals which strongly relativize $\Delta^1_1$. Also, we emphasize
that Theorem \ref{thm:hyp} (ii) is true for any 
recursively presented Polish space $\mathcal{X}$, not
just the effectively zero-dimensional ones. 

The next result is an application of Theorem \ref{thm:hyp}, which we will
use in the next section to point out an interesting connection
between hyperlow reals and computably dominated reals.

\begin{theorem}\label{thm:totalhyp}
For $f\in \baire$, the following are equivalent:
\begin{enumerate}[(i)]
\item $f$ is hyperlow;
\item For every $g\in \baire$, $g\in \Delta^1_1(f)$ if and
only if there is a $\Delta^1_1$ total function $F:\baire \rightarrow \baire$
such that $F(f)=g$. 
\end{enumerate}
\end{theorem}

\begin{proof}
Suppose $f$ is hyperlow and $g\in \Delta^1_1(f)$.
By Theorem \ref{thm:hyp}, $\Delta^1_1(f)=\mathcal{S}\Delta^1_1(f)$, so there
is a $\Delta^1_1$ relation $P\subseteq \baire \times \Nats^2$ such that
\[
g(n)=k \iff P(f, n, k).
\]
Now, we define $F:\baire\rightarrow \baire$ by
\begin{multline*}
F(h)(n)=k \iff [(\forall \ell)\neg P(h, n, \ell) \ \& \ k=0] \\
\vee 
[P(h, n, k) \ \& \ (\forall \ell<k)\neg P(h, n, k)].
\end{multline*}
Easily, $F$ is $\Delta^1_1$ and $F(f)=g$.

Now assume (ii) holds and let $B\subseteq \N$ be $\Delta^1_1(f)$.
Let $F:\baire \rightarrow \baire$ be a $\Delta^1_1$ total 
function
with $F(f)= B$. Then, we easily have the equivalence
\[
n\in B \iff F(f)(n)=1,
\]
which shows that $B$ is $\mathcal{S}\Delta^1_1(f)$.
\end{proof}

We close with some remarks on strongly relativizing 
the higher analytical
self-dual pointclasses $\Delta^1_n$, $2\leq n<\omega$.
It is easy to show that $\Delta^1_n$-singletons 
strongly relativize $\Delta^1_n$, but we would like
to know if there are more interesting examples 
(and to know whether there
are uncountably many such).  
The answers to these questions may be independent of
\textsf{ZFC}.  
Projective Determinacy implies that every $\Pi^1_{2n+1}$
has the reduction property.  This property is enough
to construct a good coding for $\Delta^1_{2n+1}$ subsets
of any space $\mathcal{X}$.  These were the only
tools we used in our proof of Theorem \ref{thm:hyp},
so the following immediately follows from that proof.

\begin{theorem} \label{thm:analytical}
    Assume Projective Determinacy and let $f\in \baire$
    \begin{enumerate}[(i)]
    \item Suppose $f\in \baire$ has the following property:
    for every $\Pi^1_{2n+1}$ set $P\subseteq \baire$ with
    $f\in P$, there is a $\Delta^1_{2n+1}$ set $Q$ with
    $f\in Q\subseteq P$.  Then, $f$
    strongly relativizes $\Delta^1_{2n+1}$ (at every 
    space $\mathcal{X}$).
    \item If $f$ strongly relativizes $\Delta^1_{2n+1}$ at $\N$, then $f$ does not compute the
    the $\Pi^1_{2n+1}$-complete set of integers.
    \end{enumerate}
\end{theorem}

In the case of $\Delta^1_1$, the hypotheses on $f\in \baire$
in (i) and (ii) of Theorem \ref{thm:analytical} are
both equivalent to being hyperlow.  For the higher
analytical pointclasses, we ask the following question.

\begin{question}
    Is there a nice characterization (perhaps under Projective Determinacy) of the reals
    which strongly relativize $\Delta^1_{n}$ for $n\geq 1$?  Are there uncountably many such reals?
\end{question}

\section{Strongly relativizing $\Delta^0_1$}

\label{sec:recursiveatn}

The property of strongly relativizing $\Delta^0_1$ at $\N$
is closely connected with computable domination of functions.
A real $f\in \baire$ is \textbf{computably dominated}
if for every $g\in \baire$ with $g\leq_T f$, there is a computable $h\in \baire$ which dominates $g$ pointwise, i.e., such that $g(n)\leq h(n)$ for all $n\in \N$. Most of the results in this section will focus
on $A\in 2^\N$ which are computably dominated.

Martin and Miller established the key facts about
computably dominated reals, of which we will make use
of the following.

\begin{theorem}[\cite{miller1968}]\label{martinmiller}
\begin{enumerate}[(i)]
\item There are
continuum-many $A\in 2^\N$ which are computably dominated.  Moreover, this result relativizes
to any $B\in 2^\N$, i.e., there are continuum-many reals
$A>_T B$
which are $B$-computably dominated in the sense that
every $A$-computable function from $\N$ to $\N$ is dominated
pointwise by some $B$-computable function.
\item If $A\in 2^\N$ is a noncomputable real which is Turing comparable with $\emptyset'$,
then $A$ is not computably dominated.  
\end{enumerate}
\end{theorem}

For elements of Cantor space, computable domination can be characterized in terms of 
truth-table reductions.  For $A, B\in 2^\N$, 
$B$ is \textbf{truth-table reducible} to $A$, denoted $B\leq_{tt} A$, if there is a computable total
$f:\N\rightarrow \N$ and a computable $R\subseteq 2^{<\N}\times \N$
such that
\begin{equation}\label{ttdef}
n\in B \iff R(A\restriction f(n), n)
\end{equation}
In other words, given $n$ we can compute a truth table, where $f(n)$ tells us how many oracle bits we have to look at, and for each row of the table $R$ tells whether we can conclude that $n\in B$.

There is another useful characterization of truth-table reducibility.
For $A, B\in 2^\N$, $B$ is Turing reducible to $A$, denoted $B\leq_T A$,
    if and only if there is a computable \textit{partial}
    $F: 2^\N \rightharpoonup 2^\N$ with $F(A)=B$.  When
    $B\leq_{tt} A$, we can say something stronger:

\begin{theorem}[Nerode; see \cite{soare_2016}, Theoem 3.8.5]
    $B\leq_{tt} A$ if and only if there is a computable
    \textit{total}
     $F:2^\N\rightarrow 2^\N$ with
    $F(A)=B$. 
\end{theorem}

$A\in 2^\N$ is called \textbf{$tt$-irreducible} if $B\leq_T A$ implies
$B\leq_{tt} A$ for every $B\in 2^\N$.  
Compare this with Theorem \ref{thm:totalhyp}, which says
that $A$ being hyperlow is equivalent to the existence of a
total $\Delta^1_1$ function mapping $A$ to $B$ whenever $B\in \Delta^1_1(A)$. 
Thus,
hyperlowness is, in some sense, a $\Delta^1_1$ analog of being
$tt$-irreducible.  We will make the analogy even stronger,
in Theorem \ref{thm:relopenatn}, by showing
that both properties are connected to strongly 
relativizing particular self-dual Kleene pointclasses.
First, we point out the following equivalent characterization
of $tt$-irreducible reals:

\begin{theorem}[See \cite{soare_2016}, Theorem 5.6.4]
$A\in 2^\N$ is computably dominated if and only if $A$ is
$tt$-irreducible.
\end{theorem}

Note that, with the above definitions, it does not
make sense to talk about $f\in \baire\setminus 2^\N$ being
$tt$-irreducible, but we can talk about $f$ being computably
dominated.

The following lemma establishes the essential connection
between truth-table reductions and the class
$\mathcal{S}\Delta^0_1(A)\restriction \N$.

\begin{lemma} Let $A, B\in 2^\N$.  
$B$ is $\mathcal{S}\Delta^0_1(A)$ if
and only if $B\leq_{tt} A$.
\end{lemma}

\begin{proof}
    $(\Leftarrow)$ is immediate from the definition
    of $tt$-reduction, specifically equivalence (\ref{ttdef}).

    $(\Rightarrow)$ Suppose there is a computable
    $P\subseteq 2^\N\times \N$ such that
    \[
    n\in B \iff P(A, n)
    \]
    Pick  $P_0, P_1\subseteq 2^{<\N}\times \N$ which are 
    computable, closed upwards in the first argument,
    and satisfy
    \begin{align*}
    P(C, n) &\iff (\exists k)P_0(C\restriction k, n), \\
    \neg P(C, n) &\iff (\exists k) P_1(C\restriction k, n)
    \end{align*}
    for every $C\in 2^\N$ and every $n\in \N$. 
    By compactness of $2^\N$, 
    for every $n\in \N$ there is a $k$ such that 
    (exactly) one of $P_i(\sigma, n)$, $i=0, 1$, holds
    for every length $k$ string $\sigma$.
    Moreover, if we let $f(n)$ be the least such $k$,
    then the function $f:\N\rightarrow \N$ is clearly computable.  Thus,
    \[
    n\in B \iff P_0(A\restriction f(n), n),
    \]
    which shows $B\leq_{tt} A$
\end{proof}

\begin{theorem} \label{thm:relopenatn}
$A\in 2^\N$ strongly relativizes
 $\Delta^0_1$ at $\N$ if and only if $A$ is
computably dominated.  In particular, there
are continuum-many reals which strongly 
relativize $\Delta^0_1$ at $\N$.  
\end{theorem}

\begin{proof}
$A$ strongly relativizes
$\Delta^0_1$ at $\N$ 
iff every $\Delta^0_1(A)$ set is $\mathcal{S}\Delta^0_1(A)$
iff for every $B\in 2^\N$, $B\leq_{T} A$ implies $B\leq_{tt} A$
iff $A$ is $tt$-irreducible iff $A$ is computably dominated.
\end{proof}

 It is now natural to ask whether this characterization
 of strongly relativizing $\Delta^0_1$ at $\N$
 extends to elements of Baire space.  We
 will see in Theorem \ref{thm:turingnopreserve} that it does not.
 There are indeed $f\in \baire$ which strongly
 relativize $\Delta^0_1$ at $\N$ but which
 are not computably dominated.  
 In particular, this means that
 strongly relativizing $\Delta^0_1$ at $\N$
 is not a property that is preserved under
 Turing equivalence (in $\baire$).  
As seen in the next proposition, a slightly
stronger condition than Turing equivalence
does preserve strongly relativizing $\Delta^0_1$ at
any space.

 \begin{proposition}\label{prop:transferopen}
Let 
$\mathcal{X}$ be a space, and let
$f, g\in \baire$.
    Suppose  $g\leq_T f$ and there is a computable \textit{total} $F:\baire \rightarrow \baire$ with $F(g)=f$.  Then, if
        $f$ strongly relativizes $\Delta^0_1$ at $\mathcal{X}$,
        then $g$ strongly relativizes $\Delta^0_1$ at $\mathcal{X}$. 
\end{proposition}

\begin{proof} 
Suppose $f$ strongly relativizes $\Delta^0_1$ at $\mathcal{X}$.
Let $P\subseteq \mathcal{X}$ be $\Delta^0_1(g)$.  Since $g\leq_T f$,
$P$ is also $\Delta^0_1(f)$.  Since $f$ strongly relativizes 
$\Delta^0_1$ at $\mathcal{X}$, there is a $\Delta^0_1$ set $Q\subseteq \baire\times \mathcal{X}$ such that $P=Q_f$.  Define $R\subseteq \baire\times \mathcal{X}$ by
\[
R(h, x) \iff  Q(F(h), x),
\]
which clearly has $P=R_g$.  Moreover, $R$
is in $\Delta^0_1$ since $\Delta^0_1$ is closed
under computable \textit{total} substitutions.  
\end{proof}

We can use this proposition to show that
one implication from Theorem \ref{thm:relopenatn}
holds for elements of Baire space.

\begin{theorem}
    \label{thm:openatn}
    If $f\in \baire$ is computably dominated, then
    it strongly relativizes $\Delta^0_1$ at $\N$.  
\end{theorem}

\begin{proof}
    Let $f\in \baire$ be computably dominated.
    Fix a computable bijection $(n, m)\mapsto \langle n, m\rangle$ from $\N^2$ to $\N$.  For $h\in \baire$, let 
    $G_h=\{\langle n, k\rangle\in \N : h(n)=k\}$.
    The total function $F:\baire \rightarrow 2^\N$, $F(h):=G_h$, is
    clearly computable.  Now, $F(f)=G_f$ and we also have $G_f\equiv_T f$.
    So, $G_f$ is computably dominated, hence
    strongly relativizes $\Delta^0_1$ at $\N$ by
    Theorem \ref{thm:relopenatn}. 
    It now follows by Proposition \ref{prop:transferopen} that $f$
    strongly relativizes $\Delta^0_1$ at $\N$.  
\end{proof}

Now, we turn our attention to $\Delta^0_1$
subsets of uncountable zero-dimensional spaces.  Since
$2^\N$ is compact, every real 
strongly relativizes $\Delta^0_1$ at $2^\N$.
Therefore, we focus on strongly relativizing
$\Delta^0_1$ subsets of $\baire$.

Our examples of strongly relativizing reals for subsets
of $\baire$ come from definable singletons.  
If $\Gamma$ is a pointclass, we call $f\in \baire$
a \textbf{$\Gamma$-singleton} if the singleton
$\{f\}$ is a $\Gamma$ subset of $\baire$.  

\begin{theorem}\label{thm:pi01singleton}
    If $f\in \baire$ is a $\Pi^0_1$-singleton,
    then $f$ strongly relativizes $\Delta^0_1$
    at $\baire$.
\end{theorem}

\begin{proof}
    Pick a computable tree $T$ on $\N$ such
    that $[T]=\{f\}$.  Let $P\subseteq \baire$
    be $\Delta^0_1(f)$.  Choose computable
    $R_0, R_1\subseteq (\N^{<\N})^2$, closed upwards
    in both arguments, such that
    \begin{align*}
    P(g) &\iff (\exists n)R_0(f\restriction n, g\restriction n)\\
    \neg P(g) &\iff (\exists n)R_1(f\restriction n, g\restriction n)
    \end{align*}
    for all $g\in \baire$.  Now, define $Q\subseteq \baire^2$ by 
    \[
    Q(h, g) \iff (\exists n)[h\restriction n\in T
    \ \& \ R_0(h\restriction n, g \restriction n) \ \& \ (\forall i<n)\neg R_1(h\restriction i, g \restriction i)]
    \]
    Immediately, $Q$ is $\Sigma^0_1$ and $P=Q_f$ since
    $f\restriction n \in T$ for all $n\in \N$.
    However, we can also show $Q$ is $\Pi^0_1$ by establishing the
    equivalence
    \begin{multline} \label{notq}
        \neg Q(h, g) \iff (\exists n)[h\restriction n\notin T \ \& \ (\forall i<n)\neg (R_0(h\restriction i, g \restriction i) \ \vee R_1(h\restriction i, g \restriction i))] \\
        \vee (\exists n)[h\restriction n\in T
    \ \& \ R_1(h\restriction n, g \restriction n) \ \& \ (\forall i\leq n)\neg R_0(h\restriction i, g \restriction i)].
    \end{multline}
    Once we show this equivalence is true, we have
    that $Q$ is $\Delta^0_1$ and $P=Q_f$ is
    $\mathcal{S}\Delta^0_1(f)$. 
    
    The ($\Leftarrow$) direction of equivalence (\ref{notq}) is clear, so we only need to show
    $(\Rightarrow)$.  Suppose $\neg Q(h, g)$ holds.  
    The implication is clearly true when $h=f$, so suppose
    $h\neq f$.   We examine two cases.
    In the first case, there is $n$ such that $h\restriction n\in T$
    and at least one of $R_j(h\restriction n, g\restriction n)$, $j=0, 1$, holds.  Consider the least such $n$. Because $\neg Q(h, g)$ holds, it cannot be that $R_0(h\restriction n, g\restriction n)$ is true.  Thus, this $n$ witnesses that the second disjunct 
    (\ref{notq}) holds. 
    
    If the first case fails, look at the least $n$ with $h\restriction n\notin T$ (which must exists since $h\neq f$).
    Since the first case fails, both $R_0(h\restriction i, g\restriction i)$ and $R_0(h\restriction i, g\restriction i)$ must fail for all $i<n$.   It follows that $n$ witnesses
    that the first disjunct of (\ref{notq}) is true. 
    Thus, we have proven $(\Rightarrow)$ for equivalence (\ref{notq}).
\end{proof}

We make several remarks about this result:

(1) The only $A\in 2^\N$ which are $\Pi^0_1$-singletons are the computable elements of $2^\N$.
So, Theorem \ref{thm:pi01singleton} is
only interesting for elements
of $\baire\setminus 2^\N$.  

(2) There are many interesting, nontrivial $\Pi^0_1$-singletons.  In fact,
every $\Delta^1_1$ real is computed
by a $\Pi^0_1$-singleton; see \cite[4A.8]{moschovakis2009}.
Moreover, every $\emptyset^{(\alpha)}$
is Turing equivalent to a $\Pi^0_1$-singleton
$f\in \baire\setminus 2^\N$; see \cite[Theorem 2.1.4 and Proposition 2.1.5]{chong2015}.

(3) On the other hand, not every $\Delta^1_1$
real is a $\Pi^0_1$-singleton; see \cite{feferman_1965}.
This leaves us with the following question.

\begin{question}
Are there non-$\Delta^1_1$ reals which
strongly relativize $\Delta^0_1$ at $\baire$?
Are there continuum-many?
\end{question}

\begin{remark} \label{remark:lutz}
    A natural idea to find more reals which strongly
    relativize $\Delta^0_1$ at $\baire$ is to look
    at the higher-type analog of computable domination.
    Say a real $A\in 2^\N$ is \textbf{computably Baire-dominated}
    if every total $F:\baire \rightarrow \N$ which is 
    $A$-computable is pointwise dominated by a
    computable total $G:\baire \rightarrow \N$.  
    It is not difficult to show that any real which
    is computably Baire-dominated must strongly
    relativize $\Delta^0_1$ at $\baire$.  However,
    Patrick Lutz proved that the only computably Baire-dominated
    reals are the computable reals.  Lutz's proof is
    given in the appendix.  
\end{remark}

Now, we can show that Turing equivalence does
not preserve strongly relativizing $\Delta^0_1$
at $\N$ or at $\baire$.  

\begin{theorem}\label{thm:turingnopreserve}
    There exists $f\in \baire$ and $A\in 2^\N$
    such that $f$ strongly
    relativizes $\Delta^0_1$ at $\baire$
    and $f\equiv_T A$, but $A$ does not
    strongly relativize $\Delta^0_1$ at $\baire$, or even at $\N$.
    In particular, there exists $f\in \baire$
    which strongly relativizes $\Delta^0_1$
    at $\baire$ but is not computably dominated.  
\end{theorem}

\begin{proof}
    Let $A:=\emptyset'$ and let $f\in \baire$
    be a $\Pi^0_1$-singleton which is Turing
    equivalent to $\emptyset'$.  By Theorem \ref{thm:pi01singleton}, $f$ strongly relativizes $\Delta^0_1$ at $\baire$. 
    However, $\emptyset'$ is not 
    computably dominated, so it does
    not strongly relativize $\Delta^0_1$
    at $\N$ by Theorem \ref{thm:relopenatn}.
    By Corollary \ref{cor:reciso}, $\emptyset'$ does not
    strongly relativize $\Delta^0_1$ at $\baire$.
    Also note that $f$ is not computably
    dominated since $f\equiv_T \emptyset'$.  
\end{proof}

\begin{question}
    Are there reals which strongly relativize $\Delta^0_1$
    at $\N$ but do not strongly relativize $\Delta^0_1$
    at $\baire$?
\end{question}

\section{Strongly relativizing $\Delta^0_{1+\alpha}$}

\label{sec:arithatn}

We now turn our attention to the self-dual effective
Borel classes of higher order, $\Delta^0_{1+\alpha}$.
We begin by pointing out that for most of these classes,
strongly relativizing them is a property
preserved under Turing equivalence.

\begin{proposition}\label{turingclassespart2}
Let $\Delta$ be a self-dual Kleene pointclass other
than $\Delta^0_1$ and $\Delta^0_2$, let 
$\mathcal{X}$ be a space, and let
$f, g\in \baire$ be Turing equivalent.
Then,
 $f$ strongly relativizes
        $\Delta$ at $\mathcal{X}$ iff $g$ strongly relativizes $\Delta$ at $\mathcal{X}$.
\end{proposition}

\begin{proof} 
Suppose $f\equiv_T g$ and that $f$ strongly relativizes $\Delta$ at $\mathcal{X}$.
Since $f\leq_T g$, there is $e\in \N$ with $f=\Phi_e^g$.  Let  $F:\baire \rightharpoonup \baire$ be the Turing functional
given by $\Phi_e$.  It is easy to check
that $\text{dom}(F)$ is $\Pi^0_2$, and hence in $\Delta$.  
Let $P\subseteq \mathcal{X}$ be $\Delta(g)$.  Since $g\leq_T f$,
$P$ is also $\Delta(f)$.  Since $f$ strongly relativizes 
$\Delta$ at $\mathcal{X}$, there is a $\Delta$ set $Q\subseteq \baire\times \mathcal{X}$ such that $P=Q_f$.  Define $R\subseteq \baire\times \mathcal{X}$ by
\[
R(h, x) \iff F(h)\downarrow \ \& \ Q(F(h), x),
\]
which is clearly in $\Delta$ and satisfies $P=R_g$.
\end{proof}

Theorem \ref{thm:turingnopreserve} established
that strongly relativizing $\Delta^0_1$ is not
preserved by Turing equivalence.  This leaves
the following question.

\begin{question}
    Does Turing equivalence preserve
    the property of strongly relativizing
    $\Delta^0_2$ (at any space)?
\end{question}

Next, we will transfer results about strongly relativizing
$\Delta^0_1$ to the classes $\Delta^0_\alpha$, $1<\alpha<\omegack$, using relativization
techniques which relate these classes to each
other.  
The $\Delta^0_{1+\alpha}$ subsets of $\N$ are related
to relativized versions of $\Delta^0_1$ by the following
standard fact.

\begin{theorem}\label{keyrelfact}
For any $f\in \baire$ and any computable ordinal $\alpha$,
\[
\Delta^0_{1+\alpha}(f) \restriction \N = 
\Delta^0_1(f^{(\alpha)}) \restriction \N.
\]
\end{theorem}

We comment on the need for the $1+\alpha$.  When $\alpha=n\in \N$, we have $1+n=n+1$ and the theorem says that
$\Delta^0_{n+1}(f)$ subsets of $\N$ are the same as
the $f^{(n)}$-computable subsets.  For infinite $\alpha$,
we have
$1+\alpha = \alpha$, so the $\Delta^0_{\alpha}(f)$ subsets
of $\N$ are exactly the $f^{(\alpha)}$-computable ones.
This case distinction comes from the tradition of
calling the computable sets $\Delta^0_1$ instead of $\Delta^0_0$, and the $1+\alpha$ gimmick helps us
avoid having to write the cases separately.   

Relativizing sets of reals is a bit more complicated, as 
described by the following standard fact from higher-order
computability theory.  

\begin{theorem}\label{bairenormalform}
    $P\subseteq \baire\times \N$ is $\Sigma^0_{1+\alpha}$ if and only if there is
    a $\Sigma^0_1$ relation $R\subseteq \baire\times \N$ such that
    \[
    P(g, n) \iff R(g^{(\alpha)}, n)
    \]
    for all $g\in \baire$ and $n\in \N$.  Moreover,
    for $f\in \baire$, 
    $P\subseteq \baire\times \N$ is $\Sigma^0_{1+\alpha}(f)$ if and only if there is a $\Sigma^0_1$ relation
    $R\subseteq \baire\times \N$ such that
    \[
    P(g, n) \iff R((f\oplus g)^{(\alpha)}, n)
    \]
    for all $g\in \baire$ and $n\in \N$. 
\end{theorem}

Theorem \ref{keyrelfact} and Theorem \ref{bairenormalform} are
both proved for subsets of all recursively presented Polish
metric spaces in \cite{thesis}.

The results and proofs of Section \ref{sec:recursiveatn} on strongly relativizing
$\Delta^0_1$ at $\N$
do not completely relativize to $\Delta^0_{1+\alpha}$,
mostly due to the nature of relativizing subsets
of $\baire$ as described in Theorem \ref{bairenormalform}.
One manifestation of this is that
the analog of Theorem \ref{keyrelfact} for the $\mathcal{S}\Delta^0_{1+\alpha}$
classes is not true in general.  Since $\Delta^0_{1+\alpha}(f)\restriction \N =
\Delta^0_1(f^{(\alpha)})\restriction \N$, one might hope that the equality
remains true when you place an $\mathcal{S}$ in front of both classes.  While one of the inclusions is true, the other is false
in general.

\begin{proposition}
Let $0<\alpha<\omegack$ and $f\in \baire$.
\begin{enumerate}[(i)]
\item $\mathcal{S}\Delta^0_1(f^{(\alpha)})\restriction \N\subseteq \mathcal{S}\Delta^0_{1+\alpha}(f)\restriction \N$.
    \item If $f$ strongly relativizes $\Delta^0_{1+\alpha}$ at $\N$,
    then 
    \[
    \mathcal{S}\Delta^0_1(f^{(\alpha)})\restriction \N\subsetneq \mathcal{S}\Delta^0_{1+\alpha}(f)\restriction \N.
    \]
    \end{enumerate}
\end{proposition}

\begin{proof}
    (i) Suppose $B\subseteq \N$ is $\mathcal{S}\Delta^0_1(f^{(\alpha)})$. Pick a computable
    $R\subseteq \baire\times \N$ with
    \[
    n\in B \iff R(f^{(\alpha)}, n)
    \]
    By Theorem \ref{bairenormalform}, $Q=\{(g, n)\in \baire\times \N : R(g^{(\alpha)}, n)\}$ is $\Delta^0_{1+\alpha}$ and
    clearly $B=Q_f$.
    
    (ii) Assume $f$ strongly relativizes $\Delta^0_{1+\alpha}$ at $\N$.  Since $\alpha>0$, $\emptyset'\leq_T f^{(\alpha)}$ and $f^{(\alpha)}\in 2^\N$. It follows from 
    Theorem \ref{martinmiller} that $f^{(\alpha)}$ is not computably dominated.  Thus, by
    Theorem \ref{thm:relopenatn}, 
    $f^{(\alpha)}$ does not strongly relativize $\Delta^0_1$ at $\N$.  
    Pick $B\subseteq \N$ which is $\Delta^0_1(f^{(\alpha)})$ but not
    $\mathcal{S}\Delta^0_1(f^{(\alpha)})$.  
    Since $f$ strongly relativizes $\Delta^0_{1+\alpha}$ at $\N$,
    \[
    \Delta^0_1(f^{(\alpha)})\restriction \N=\Delta^0_{1+\alpha}(f)\restriction \N=\mathcal{S}\Delta^0_{1+\alpha}(f)\restriction \N.
    \]
    Thus, $B$ is in $\mathcal{S}\Delta^0_{1+\alpha}(f)$
    but not in $\mathcal{S}\Delta^0_1(f^{(\alpha)})$
\end{proof}

Next, we will establish a sufficient condition for a real
to strongly relativize $\Delta^0_{1+\alpha}$ at $\N$,
but first we need some definitions.  
Let $\mathcal{X}$ be a space and let $\Gamma$ be a 
pointclass. A total
function $f:\mathcal{X}\rightarrow \N$ is a \textbf{$\Gamma$ function} if its graph relation
$\{(x, n)\in \mathcal{X}\times\N: f(x)=n\}$ is
in $\Gamma$.  It is an easy exercise to show
a function $f:\mathcal{X}\rightarrow \N$ is a
$\Sigma^0_{\alpha}$ function if and only if it is a $\Delta^0_\alpha$
function.  

\begin{definition}
    Let $\alpha$ be a computable ordinal.
    $f\in \baire$ is $\Delta^0_{1+\alpha}$-dominated if every
    total $\Delta^0_{1+\alpha}(f)$ function $g:\N\rightarrow \N$
    there is a total $\Delta^0_{1+\alpha}$ function
    $h:\N\rightarrow \N$ with $g(n)\leq h(n)$ for all $n\in\N$. 
\end{definition}

We note the following helpful alternative characterization
of being $\Delta^0_{1+\alpha}$-dominated.  

\begin{proposition} \label{deltadominated}
    Let $\alpha$ be a computable ordinal.
    $f\in\baire$ is $\Delta^0_{1+\alpha}$-dominated if and only if
    $f^{(\alpha)}$ is computably dominated relative to $\emptyset^{(\alpha)}$.
\end{proposition}

\begin{proof}
    This follows immediately from the fact that for
    relations on $\N$, $\Delta^0_{1+\alpha}(f)=\Delta^0_1(f^{(\alpha)})$ and
    $\Delta^0_{1+\alpha} = \Delta^0_1(\emptyset^{(\alpha)})$.
\end{proof}

We will prove that being $\Delta^0_{1+\alpha}$-dominated
is sufficient to conclude strongly relativizing
$\Delta^0_{1+\alpha}$ at $\N$, but first we point
out the following standard fact that we will use multiple
times.

\begin{lemma} \label{bddsearch}
    Let $\alpha<\omegack$.  If $Q\subseteq \mathcal{X}\times\N$ is $\Delta^0_{1+\alpha}$ and $g:\mathcal{X}\rightarrow \N$ is a $\Delta^0_{1+\alpha}$ function, then the relation $R\subseteq \mathcal{X}$ defined by
    \[
    R(x) \iff (\exists k\leq g(x))~Q(x, k)
    \]
    is also $\Delta^0_{1+\alpha}$.  
\end{lemma}

\begin{proof}
    This follows from the equivalences
    \begin{align*}
    R(x) &\iff (\exists k, i)[g(x)=i \ \& \ k\leq i \ \& \ Q(x, k)] \\
    & \iff (\forall i)[g(x)=i \implies (\exists k\leq i)Q(x, k)],
    \end{align*}
    along with the standard closure properties for
    $\Delta^0_{1+\alpha}$.  
\end{proof}

\begin{theorem}\label{suffatn}
    Let $\alpha<\omegack$ and let $f\in \baire$.
    If $f$ is $\Delta^0_{1+\alpha}$-dominated, then
    $f$ strongly relativizes $\Delta^0_{1+\alpha}$ at $\N$.
\end{theorem}

\begin{proof}
    Suppose $f$ is $\Delta^0_{1+\alpha}$-dominated
    and let $B\in2^\N$ be $\Delta^0_{1+\alpha}(A)$.  Then,
    $B$ is $\Delta^0_1(f^{(\alpha)})$, and so we can 
    find computable $P_0, P_1\subseteq \N^{<\N}\times \N$ which satisfy
    \[
    n\in B \iff (\exists k)P_0(f^{(\alpha)}\restriction k, n)
    \]
    \[
    n\notin B \iff (\exists k)P_1(f^{(\alpha)}\restriction k, n)
    \]
    Define $g:\N\rightarrow \N$ by letting $g(n)$ be
    the least $k$ such that exactly one of 
    ${P_i(f^{(\alpha)}\restriction k, n)}$, $i=0, 1$, holds.  This function is $f^{(\alpha)}$-computable,
    hence it is $\Delta^0_{1+\alpha}(f)$.  Since $f$
    is $\Delta^0_{1+\alpha}$-dominated, there is
    a $\Delta^0_{1+\alpha}$ function $h$ which bounds $g$
    pointwise.  Now, define $Q\subseteq \N^\N\times N$ by
    \[
    Q(u, n) \iff (\exists k\leq h(n))P_0(u^{(\alpha)}\restriction k, n)
    \qquad (u\in \baire, \ \ n\in \N).
    \]
    By Lemma \ref{bddsearch}, $Q$ is $\Delta^0_{1+\alpha}$.
    It is also clear that $B=Q_f$. 
\end{proof}

Next, we show that there are lots of reals which are
$\Delta^0_{1+\alpha}$-dominated and, hence, strongly
relativize $\Delta^0_{1+\alpha}$ at $\N$.
To do this, we will use the results of Martin and Miller
together with the following extension of the
Friedberg jump inversion theorem due to MacIntyre.

\begin{theorem}[\cite{macintyre_1977}] \label{jumpinversion}
   If $1\leq \alpha<\omegack$ and $A\geq_T \emptyset^{(\alpha)}$, then there is
   $B\in 2^\N$ such that $B^{(\alpha)}\equiv_T A$
\end{theorem}

\begin{theorem}
    \label{thm:arithatn}
    There are continuum-many reals which
    are $\Delta^0_{1+\alpha}$-dominated.  Hence,
    there are continuum-many reals which 
    strongly relativize $\Delta^0_{1+\alpha}$ at $\N$.  
\end{theorem}

\begin{proof}
    By Theorem \ref{martinmiller}, there are continuum-many $A\in 2^\N$
    such that $A>_T \emptyset^{(\alpha)}$ and $A$ is computably dominated relative to $\emptyset^{(\alpha)}$.  Of course, since Turing degrees are countable,
    there are continuum many such $A$ with distinct
    Turing degrees.  By Theorem \ref{jumpinversion}, for each such $A$ there is $B$
    such that $B^{(\alpha)}\equiv_T A$.  Since they
    have distinct $\alpha$-jumps, these $B$'s have
    distinct Turing degrees, so there are continuum
    many.  By Proposition
    \ref{deltadominated}, each $B$ is
    $\Delta^0_{1+\alpha}$-dominated.
\end{proof}

We will see later in Theorem \ref{thm:alphadomnotsuff} that being $\Delta^0_{1+\alpha}$-dominated is not a necessary condition.
Next, we will establish a sufficient condition for strongly relativizing
$\Delta^0_{1+\alpha}$ at spaces other than $\N$.  

\begin{theorem}\label{thm:pi0alphasingleton}
    Let $\alpha>0$ be a computable ordinal.  
    If $f\in \baire$ is a $\Pi^0_{1+\alpha}$-singleton, then $f$ strongly relativizes
    $\Delta^0_{1+\alpha}$ (at every space).  
\end{theorem}

\begin{proof}
    Suppose $f\in \baire$ is a $\Pi^0_{1+\alpha}$-singleton.  We begin by showing that
    $f$ strongly relativizes $\Delta^0_{1+\alpha}$
    at $\baire$. Pick a $\Pi^0_1$ set $C\subseteq \baire$ such
    that for all $g\in \baire$,
    \[
    g=f \iff C(g^{(\alpha)})
    \]
    Let $T$ be a computable tree such that $[T]=C$.

    Let $P\subseteq \baire$ be 
    $\Delta^0_{1+\alpha}(f)$.  
    Then, we can pick computable $R_0, R_1\subseteq \N^{<\N}$, both closed upwards, 
    such that
    \[
    P(g) \iff (\exists n)R_0((f\oplus g)^{(\alpha)}\restriction n),
    \]
    \[
    \neg P(g) \iff (\exists n)R_1((f\oplus g)^{(\alpha)}\restriction n).
    \]
    Now, define $Q\subseteq \baire^2$ by
    \begin{multline*}
    Q(h, g) \iff (\exists n)[h^{(\alpha)}\restriction n \in T \ \& \ 
    R_0((h\oplus g)^{(\alpha)}\restriction n) \\ 
    \ \& \ (\forall i<n) \neg R_1((h\oplus g)^{(\alpha)}\restriction i)]
    \end{multline*}
    Then, the proof proceeds in the same manner 
    as in 
    Theorem \ref{thm:pi01singleton}, showing
    that $Q$ is $\Delta^0_{1+\alpha}$ and
    that $P=Q_f$.  

    So, we have shown that $f$ strongly 
    relativizes $\Delta^0_{1+\alpha}$ at $\baire$.
    To extend this to subsets 
    of any space $\mathcal{X}$, we only need
    to use the fact that the representation
    of $\Sigma^0_{1+\alpha}$ sets, $\alpha>0$, we used
    extends to every $\mathcal{X}$.  This is proved in \cite{thesis}.  We quickly describe the notion
    of Turing jump used there.  For $x\in \mathcal{X}$, we define
    the Turing jump of $x$ to be
    \[
    J(x) := \{e\in \N : G(e, x)\},
    \]
    where $G\subseteq\N\times \mathcal{X}$ is some
    good $\N$-parametrization of the $\Sigma^0_1$
    subsets of $\mathcal{X}$.  We iterate this
    jump along Kleene's ordinal notations, as usual,
    and get the desired representation for $\Sigma^0_{1+\alpha}$ sets.  
\end{proof}

Now, we can prove that, for $0<\alpha<\omegack$, being $\Delta^0_{1+\alpha}$-dominated
is not necessary to strongly relativize $\Delta^0_{1+\alpha}$, even at $\N$. 

\begin{theorem}\label{thm:alphadomnotsuff} 
    For any $0<\alpha<\omegack$, there is
    $f\in \baire$ which strongly 
    relativizes $\Delta^0_{1+\alpha}$ at every
    space but is not $\Delta^0_{1+\alpha}$-dominated.
\end{theorem}

\begin{proof}
    Fix $0<\alpha<\omegack$.   As mentioned in
    Section \ref{sec:recursiveatn},  there is a $\Pi^0_1$-singleton $f\in \baire$ with $f\equiv_T \emptyset^{(\alpha+1)}$.  By Theorem 
    \ref{thm:pi0alphasingleton}, $f$ strongly
    relativizes $\Delta^0_{1+\alpha}$ at every space.
    Now, it is enough to show that $f$ is not computably
    dominated relative to $\emptyset^{(\alpha)}$;
    indeed, this implies that $f^{(\alpha)}$
    is not computably dominated relative $\emptyset^{(\alpha)}$, hence $f$ is
    not $\Delta^0_{1+\alpha}$-dominated by Proposition
    \ref{deltadominated}.

    Define $g:\N\rightarrow \N$ by
    \[
    g(e) = \begin{cases} (\text{least} \ k) \Phi_{e, k}^{\emptyset^{(\alpha)}\restriction k}(e)\downarrow 
    & \text{if $e\in \emptyset^{(\alpha+1)}$} \\
    0 & \text{otherwise,}
    \end{cases}
    \]
    where $\Phi_e$ is a standard effective enumeration
    of oracle machines.
    $g$ is clearly $\emptyset^{(\alpha+1)}$-computable, hence $f$-computable.  However,
    $g$ is not dominated by a $\emptyset^{(\alpha)}$-computable function $h$ , since otherwise 
    $\emptyset^{(\alpha)}$ could compute its own
    halting problem by running $\Phi_e^{\emptyset^{(\alpha)}}(e)$ for $h(e)$ steps.  
\end{proof}

The following is left open.

\begin{question}
Is there a ``natural'' condition which is necessary
and sufficient for strongly relativizing $\Delta^0_{1+\alpha}$ at $\N$?
at $\baire$?
\end{question}

The next result establishes a necessary condition
for strongly relativizing $\Delta^0_{1+\alpha}$ at $\N$.

\begin{theorem}
    If $f\in \baire$ computes $\mathcal{O}$, then 
    $f$ does not strongly relativize any $\Delta^0_{1+\alpha}$,
    $\alpha<\omegack$, at $\N$. 
\end{theorem}

\begin{proof}
    Suppose $f\in \baire$ computes $\mathcal{O}$.
    Let $G\subseteq \N\times \baire\times $ be an $\N$-parameterization
    for $\Sigma^0_{1+\alpha}$.  
    Define $D$ to be the set of $e\in \N$ such that $G_{(e)_0}=\neg G_{(e)_1}$,
    where $e\mapsto ((e)_0, (e)_1)$ is some computable bijection from 
    $\N$ to $\N^2$.  Easily, $D$ is $\Pi^1_1$, hence $f$ computes $D$.
    From here, we can replicate the diatonalization
    argument used in the proof of
    (ii)$\Rightarrow$(iii) in Theorem \ref{thm:hyp} to construct
    a subset of $\N$ which is $\Delta^0_{1+\alpha}(f)$ but not
    $\mathcal{S}\Delta^0_{1+\alpha}(f)$.  
\end{proof}

We end by pointing out one last application of these ideas.
A \textbf{nice coding} of $\Delta^0_{1+\alpha}$ subsets of $\baire$ is 
a $\Delta^1_1$ set of integers $D$ together with 
two $\Sigma^0_{1+\alpha}$ subsets $Q^+, Q^-$ of $\N\times \baire$ such that:
(1) for every $e\in D$, $Q^+ = \neg Q^-$ and (2) every $\Delta^0_{1+\alpha}$
set $P\subseteq \N$ is equal to some $Q^+_e$.

\begin{theorem}
    Let $\alpha$ be a computable ordinal.  
    There is no nice coding of $\Delta^0_{1+\alpha}$ subsets
    of $\baire$. 
\end{theorem}

\begin{proof}
    Suppose towards a contradiction that there is a nice coding of 
    $\Delta^0_{1+\alpha}$ subsets of $\baire$.  Since $\baire$ and $\baire\times \N$ are computably isomorphic, we can easily construct
    a nice coding $D, Q^+, Q^-$ of $\Delta^0_{1+\alpha}$-subsets of
    $\baire\times \N$.  Since $D$ is $\Delta^1_1$, we can 
    pick $\beta<\omegack$
    such that $D$ is computable in $\emptyset^{(\beta)}$.  Let $f\in \baire$ be a $\Pi^0_1$-singleton
    with $f\equiv_T \emptyset^{(\beta)}$.  Of course, $f$ is also
    a $\Pi^0_{1+\alpha}$-singleton, so
    $f$ strongly relativizes $\Delta^0_{1+\alpha}$ at $\N$.  However,
    since $f$ can computably enumerate the elements of $D$,
    we can again use the construction in (ii)$\Rightarrow$(iii) of Theorem \ref{thm:hyp} to build a $\Delta^0_{1+\alpha}(f)$ set of
    integers which is not $\mathcal{S}\Delta^0_{1+\alpha}(f)$.
    Contradiction.  
\end{proof}

\appendix

\section{Computably Baire-dominated reals}

In this appendix, we give Lutz's proof that
all computably Baire-dominated reals are computable.  
First, we recall the definition and prove a useful
equivalence.

\begin{definition}
    We say that $A\in 2^\N$ is 
    \textbf{computably Baire-dominated} if
    every $A$-computable total $F:\baire \rightarrow \N$ is dominated pointwise by a computable total $G:\baire \rightarrow \N$, i.e.,
    $F(f)\leq G(f)$ for all $f\in \baire$.
\end{definition}

\begin{theorem}
    $A\in 2^\N$ is computably Baire-dominated if and only
    if every $A$-computable wellfounded tree on $\N$
    is contained in a computable wellfounded tree on $\N$. 
\end{theorem}

\begin{proof}
    $(\Rightarrow)$  Suppose $A$ is computably Baire-dominated,
    and let $T$ be an $A$-computable wellfounded tree on $\N$.
    Let $F:\baire \rightarrow \N$ be defined so that
    $F(f)$ is the least natural number $n$ such that
    $f\restriction n\notin T$.  Since $F$ is $A$-computable,
    there is a computable $G:\baire \rightarrow \N$ which
    dominates $F$ pointwise.  
    Let $G=\Phi_e$.  Here, $\Phi_e$ is a Turing
    functional that takes an
    oracle $f\in\baire$, runs program $e$ (with no other inputs), and
    outputs a natural number if it converges.  Define
    \[
    S:= \{\sigma\in \N^{<\N} : (\forall k)[\Phi^\sigma_{e, |\sigma|}\downarrow = k \implies |\sigma|\leq k]\}
    \]
    $S$ is clearly a computable tree.  It is wellfounded
    since for any $f\in \baire$, arbitrarily long initial
    segments of $f$ make $\Phi_e$ converge, including
    those with length larger than $G(f)= \Phi_e^f$.  
    Finally, we show that $T\subseteq S$.  If $\sigma\in T$,
    then $G(f)=\Phi_e^f\geq F(f)> |\sigma|$ for every $f$ extending $\sigma$.  Moverover, if $\Phi_{e, |\sigma|}^\sigma \downarrow$, then it must equal $G(f)$
    for any $f$ extending $\sigma$.  Thus, $\sigma\in S$.

    $(\Leftarrow)$ Suppose every $A$-computable
    wellfounded tree on $\N$ is contained in a computable
    wellfounded tree.  Let $F:\baire\rightarrow \N$ be
    $A$-computable.  Pick $e\in \N$ such that $F=\Phi^A_e$.
    Define $T:= \{\sigma\in \N^{<\N} : \Phi_e^{A\restriction |\sigma|}(\sigma)\uparrow\}$.  $T$ is clearly an $A$-computable tree and
    it is wellfounded since $F=\Phi^A_e$ is total.  
    Let $S$ be a computable wellfounded tree with $T\subseteq S$.
    For every $f\in \baire$, let $n_f$ be the least number
    with $f\restriction n_f\notin S$, and note that
    $f\mapsto n_f$ is computable.  Now, define $G:\baire \rightarrow \N$ by 
    \[
    G(f) = \max \{ \Phi_{e, n_f}^{\tau}(f\restriction n_f) : \tau\in 2^{|n_f|} \ \& \ \Phi_{e, n_f}^{\tau}(f\restriction n_f)\downarrow \}.
    \]
    Then, $G$ is computable and pointwise dominates $F$.  
\end{proof}

Fix some reasonable way of coding trees on $\N$
with elements
of Cantor space.  Let $\text{Tr}\subseteq 2^\N$ be the set
of reals which code trees, and let $\text{WF}$ be the set of
reals which code wellfounded trees.

\begin{lemma}[Lutz] \label{app:treecodemap}
    Let $A\in 2^\N$ be computably Baire-dominated and
    suppose $F:2^\N\rightarrow 2^\N$ is a computable function
    such that $F(A)\in \text{WF}$.
    Then there is a $\Pi^0_1$ set $\mathcal{C}\subseteq 2^\N$ 
    such that
    $A\in \mathcal{C}$ and for every $B\in \mathcal{C}$, either 
    $F(B)\notin \text{Tr}$ or $F(B)\in \text{WF}$.
\end{lemma}

\begin{proof}
    Since $F$ is computable, the wellfounded tree coded by $F(A)$ is computable relative to $A$.  
    Since $A$ is computably Baire-dominated, there is a computable
    tree $T$ which contains the tree coded by $F(A)$.  
    Let $\mathcal{C}$ be the set of all $B$ such that,
    for all $\sigma\in \N^{<\N}$, if $B$ accepts $\sigma$
    as a node
    (according to the reasonable coding scheme), then
    $\sigma \in T$.  Clearly, $\mathcal{C}$ is $\Pi^0_1$
    and every $B\in \mathcal{C}$ either does not code a tree,
    or else 
    codes a tree that is a subset of the wellfounded tree $T$.
\end{proof}

\begin{lemma}[Lutz] \label{app:avoidinfinitepi01}
    $2^\N$ can be partitioned into two $\Pi^1_1$ sets $P$ and $Q$ such 
    that neither $P$ nor $Q$ contain any infinite
    $\Pi^0_1$ sets.
\end{lemma}

\begin{proof}
    We define two arithmetical sequences $A_0, A_1, \dots$ and
    $B_0, B_1, \dots$ of reals using the following recursion.
    Given $A_i$ and $B_i$ for $i<e$, check if $\varphi_e$ codes
    a binary tree with infinitely many paths.  If it does,
    then pick $A_e$ and $B_e$ to be two distinct paths
    which are different from all the $A_i, B_i$ for $i<e$. 
    By compactness, all of this is easily seen to be arithmetical.
    Now, set $P=\{A_i : i\in \N\}$ and $Q= 2^\N\setminus P$.  
\end{proof}

\begin{theorem}[Lutz]
    If $A\in 2^\N$ is computably Baire-dominated, then $A$
    is computable.
\end{theorem}

\begin{proof}
    Let $P, Q\subseteq 2^\N$ be as in Lemma \ref{app:avoidinfinitepi01}.  Without loss of generality,
    assume $A\in P$.  Since $P$ is $\Pi^1_1$, there is a
    computable $F:2^\N\rightarrow 2^\N$ such that,
    for all $B\in 2^\N$, $F(B)\in \text{Tr}$ and, moreover, 
    $B\in P$ if and only if $F(B)\in \text{WF}$. By Lemma \ref{app:treecodemap},
    there is a $\Pi^0_1$ set $\mathcal{C}\subseteq 2^\N$
    such that $A\in P$ and for all $B\in \mathcal{C}$, either $F(B)\notin \text{Tr}$ or $F(B)\in \text{WF}$. Since the range of $F$
    is contains in $\text{Tr}$, this implies
    $\mathcal{C}\subseteq P$.
    Since $P$ does not contain any infinite $\Pi^0_1$ sets,
    it follows that $\mathcal{C}$ is finite.  Since $A$ is
    a member of the finite $\Pi^0_1$ set $\mathcal{C}\subseteq 2^\N$, $A$ is computable.  
\end{proof}

\bibliographystyle{alpha}
\bibliography{main.bib}

@article{hkl,
     AUTHOR = {Harrington, Leo A. and Kechris, Alexander S. and 
              Louveau, Alain},
     TITLE = {A {G}limm-{E}ffros dichotomy for {B}orel equivalence relations},
   JOURNAL = {J. Amer. Math. Soc.},
  FJOURNAL = {Journal of the American Mathematical Society},
    VOLUME = {3},
      YEAR = {1990},
    NUMBER = {4},
     PAGES = {903--928},
}

@book{moschovakis2009,
    AUTHOR = {Moschovakis, Yiannis N.},
    TITLE = {Descriptive Set Theory},
    SERIES = {Mathematical Surveys and Monographs},
    VOLUME = {155},
    PUBLISHER = {American Mathematical Society, Providence},
    YEAR= {2009}
    }

@phdthesis{thesis,
    title    = {The Effective Theory of Graphs, Equivalence Relations, and Polish Spaces},
    school   = {University of California, Los Angeles},
    author   = {Tyler Arant},
    year     = {2019},
}

@book{chong2015,
    AUTHOR = {Chong, Chi Tat and Yu, Liang},
    TITLE = {Recursion Theory},
    SERIES = {De Gruyter Series in Logic and its Applications},
    VOLUME = {8},
    PUBLISHER = {De Gruyter, Berlin/Boston},
    YEAR = {2015},
    ISBN = {978-3-11-027555-1}
}

@article{miller1968, title={The Degrees of Hyperimmune Sets}, volume={14}, DOI={https://doi.org/10.1002/malq.19680140704}, number={7-12}, journal={Z. Math. Logik Grundlagen Math.}, publisher={Wiley}, author={Miller, Webb and Martin, Donald A.}, year={1968}, pages={159–166} }

@book{soare_2016, address={Berlin ; Heidelberg}, title={Turing Computability : Theory and Applications}, ISBN={9783642319327}, publisher={Springer. C}, author={Robert I. Soare}, year={2016} }

@article{macintyre_1977, title={Transfinite extensions of {F}riedberg’s completeness criterion}, volume={42}, DOI={10.2307/2272313}, number={1}, journal={J.  Symb. Log.}, author={MacIntyre, John M.}, year={1977}, pages={1–10}}

@article{feferman_1965, title={Some applications of the notions of focing and generic sets}, author={Feferman, Solomon}, journal={Fundamenta Mathematicae}, volume={56}, number={3}, year={1964}, pages={325-345}}

\end{document}